\pdfoutput=1
\documentclass{amsart}

\usepackage[alphabetic,backrefs,lite]{amsrefs}
\usepackage{hyperref}
\hypersetup{
    colorlinks=true,       % false: boxed links; true: colored links
    linkcolor=blue,          % color of internal links
    citecolor=blue,        % color of links to bibliography
    filecolor=blue,      % color of file links
    urlcolor=blue           % color of external links
}

\usepackage[all]{xy}

\numberwithin{equation}{section}

\theoremstyle{plain}
\newtheorem{lemma}{Lemma}[section]
\newtheorem{theorem}[lemma]{Theorem}
\newtheorem*{introthm}{Theorem}
\newtheorem*{introconj}{Conjecture}
\newtheorem{proposition}[lemma]{Proposition}
\newtheorem{corollary}[lemma]{Corollary}
\newtheorem{conjecture}[lemma]{Conjecture}

\theoremstyle{definition}
\newtheorem{definition}[lemma]{Definition}

\theoremstyle{remark}
\newtheorem{remark}[lemma]{Remark}
\newtheorem{notation}[lemma]{Notation}

\newcommand{\cc}{\mathbb{C}}
\newcommand{\pp}{\mathbb{P}}
\newcommand{\nn}{\mathbb{N}}
\newcommand{\zz}{\mathbb{Z}}
\newcommand{\ls}{\mathcal{H}}

\newcommand{\Osh}{{\mathcal O}}

\newcommand{\Div}{\operatorname{Div}}
\newcommand{\Pic}{\operatorname{Pic}}

\newcommand{\Bs}{\operatorname{Bs}}

\begin{document}

\title[Standard classes on the blow-up of $\pp^n$]{Standard classes on the blow-up of $\pp^n$ at points in very general position}
\author{Antonio Laface}
\address{
Departamento de Matem\'atica \newline
Universidad de Concepci\'on \newline 
Casilla 160-C \newline
Concepci\'on, Chile}
\email{antonio.laface@gmail.com}

\author{Luca Ugaglia}
\address{
Dipartimento di Matematica \newline
Politecnico di Torino\newline
C.so Duca degli Abruzzi, 24\newline
10129 Torino, Italy}
\email{luca.ugaglia@gmail.com}

\begin{abstract}
We study conjectures on the dimension of linear systems on the blow-up of $\pp^2$ and $\pp^3$ at points in very general position. We provide algorithms and Maple code based on these conjectures.
\end{abstract}
\maketitle

\section*{Introduction}
In this note we consider classes $D\in\Pic(X)$, where $X$ is the blow-up of $\pp^n$ at $r$ points in very general position. We recall that the dimensions of the cohomology groups of any line bundle whose class is $D$ do not depend on the choice of the representative. We will denote these dimensions with $h^i(D)$. Given $D:=dH-\sum_im_iE_i$, where $H$ is the pull-back of the class of a hyperplane and the $E_i$'s are the classes of exceptional divisors, it is not hard to see (Proposition~\ref{h^2}) that if $d>0$ and $m_i\geq 0$, then $h^i(D)=0$ for any $i\geq 2$. Thus by the Riemann-Roch theorem $h^0(D)-h^1(D)=\chi(D)$, where the right hand side depends only on the numerical properties of $D$. We say that $D$ is {\em non-special} if
\[
  h^0(D)~h^1(D)=0
\]
and {\em special} otherwise. If $D$ is an effective class, i.e. $h^0(D)>0$, then it is non-special if and only if $h^0(D)=\chi(D)$. Thus the {\em expected dimension} of $D$ is $\max\{\chi(D),0\}$.
The aim of this note is to discuss two conjectures about special classes in dimension 2 and 3.

In Section 1 we introduce a quadratic form on $\Pic(X)$ which will be useful to describe the action of some birational automorphisms of $X$ on $\Pic(X)$. The study of these maps, called small  modifications, will be done in Section 2, while Section 3 will be devoted to the notions of pre-standard and standard forms for a class $D$. 
In Section 4 we introduce $(-1)$-classes and we study their properties with respect to the quadratic form. Section 5 contains the proof of the equivalence of two conjectures about special classes in dimension 2, one of these conjectures is the well-known S.H.G.H. conjecture, while the other is formulated in terms of standard classes.
\begin{introthm}
Let $X$ be the blow-up of $\pp^2$ at a finite number of points in very general position. The following two statements are equivalent:
\begin{enumerate}
\item an effective class $D$ is special if and only if there exists a $(-1)$-curve $E$ such that $D\cdot E\leq -2$;
\item an effective class in standard form is non-special.
\end{enumerate}
\end{introthm}
Section 6 deals with a conjecture about special classes in dimension 3. Here we recall the following (the complete formula for $h^0(D)$ is given in Conjecture~\ref{con1}).
\begin{introconj}
Let $X$ be the blow-up of $\pp^3$ at a finite number of points in very general position and let $D:=dH-\sum_im_iE_i$ be in standard form.
\begin{enumerate}
\item If $q(D)\leq 0$, then $h^0(D)=h^0(D-Q)$.
\item If $q(D) > 0$ then $D$ is special if and only if $d<m_1+m_2-1$.
%\[
%  h^0(D)=\binom{d+3}{3}-\sum_{i=1}^r\binom{m_i+2}{3}
%  +\sum_{m_i+m_j>d+1} \binom{m_i+m_j-d+1}{3}.
%\]
\end{enumerate}
\end{introconj}
Finally Section 7 contains examples of calculation of $h^0(D)$ for some $D$.
Several Maple programs will help the reader to concretely use the described algorithms for the calculation of $h^0(D)$ according to the proposed conjectures.

\section{Basic setup}
Let us first recall some definitions and fix some notations.
\subsection{Points in very general position}
Let $p_1,\dots, p_r$ be distinct points of $\pp^n$ and let $m\in\nn^r$.
Consider the Hilbert scheme $(\pp^n)^{[r]}$ parametrizing $r$-tuples of points in $\pp^n$ and let 
$\mathcal{P}\in (\pp^n)^{[r]}$ be the point corresponding to the $p_i$'s.
%Let $\mathcal{P}$ be the point of the Hilbert scheme $(\pp^n)^{[r]}$ of $r$ points in $\pp^n$ corresponding to our $p_i$'s. 
Denote by $\ls(d,m,\mathcal{P})$ the vector space of degree $d$ homogeneous polynomials of $\cc[x_0,\dots,x_n]$ with multiplicity at least $m_i$ at each $p_i$.
Observe that $\dim\ls(d,m,\mathcal{P})$ depends on $\mathcal{P}$ and that there is an open Zariski subset $\mathcal{U}(d,m)\subseteq (\pp^n)^{[r]}$ where this dimension attains its minimal value. Let us denote by
\begin{equation}\label{U}
  \mathcal{U} := \bigcap_{(d,m)\in\nn^{r+1}}\mathcal{U}(d,m).
\end{equation}
\begin{notation}
From now on we will say that the points $p_1,\dots,p_n\in\pp^n$ are in 
{\em very general position} if the corresponding $\mathcal{P}$ is in $\mathcal{U}$ 
(which is the complementary of a countable union of Zariski closed subspaces of the configuration space).\\
%and observe that $\{p_1,\dots,p_r\}\in\mathcal{U}$ is a set of points in {\em very general position} of $\pp^n$ in the sense that they belong to the  
Moreover, given $p_1,\dots,p_r\in\pp^n$ in 
very general position, we will denote by $\pi: X\to\pp^n$  the blow-up of $\pp^n$ at the $p_i$'s, with exceptional divisors $E_1,\dots,E_r$ and by $H$ the pull-back of a hyperplane of $\pp^n$. 
\end{notation}
\begin{proposition}\label{h^2}
Let $D:=dH-\sum_im_iE_i$ with $d>0$ and $m_i\geq 0$. Then $h^i(D)=0$ for any $i>1$.
\end{proposition}
\begin{proof}
By abuse of notation denote by $H$ a general representative of the class $H$. 
Consider the exact sequence of sheaves
\[
  \xymatrix@1{
  0\ar[r] & \Osh_X\ar[r] & \Osh_X(H)\ar[r] & \Osh_{\pp^{n-1}}(1)\ar[r] & 0.
  }
\]
Since $X$ is rational, we have $h^i(\Osh_X)=0$ for $i>0$. Thus, taking cohomology, we get $h^i(H)=h^i(\Osh_{\pp^{n-1}}(1))=0$ for $i>0$. Now assume by induction that $h^i(d'H)=0$ for $d'<d$ and $i>0$. Tensoring the preceding sequence with $\Osh_X((d-1)H)$ and taking cohomology one gets $h^i(dH)=h^i(\Osh_{\pp^{n-1}}(d))=0$
for $i>0$ by~\cite[\S III Thm. 5.1]{Ha}. 

We now proceed by induction on $m:=\sum_im_i$. If $m=0$ we have already proved the statement. Suppose it is true for $m'<m$ and let us prove it for $m$. 
We can assume $m_1>0$. Consider the exact sequence of sheaves
\[
  \xymatrix@1{
  0\ar[r] & \Osh_X(D)\ar[r] & \Osh_X(D+E_1)\ar[r] & \Osh_{\pp^{n-1}}(m_1-1)\ar[r] & 0.
  }
\]
By~\cite[\S III Thm. 5.1]{Ha} we have $h^i(\Osh_{\pp^{n-1}}(m_1-1))=0$ for $i>0$.
By induction hypothesis $h^i(D+E_1)=0$ for $i>1$.  Thus we get the thesis.
\end{proof}
\begin{remark}\label{mi<0}
If $D=dH-\sum m_iE_i\in\Pic(X)$ is effective and $m_i < 0$ for some $i$, then $E_i\subset\Bs|D|$. 
%is contained in the base locus of $D$. 
In fact, if we denote by $e_i$ the class of a line in $E_i$, we have the following intersection products: $e_i E_j=-\delta_{i,j},\ e_i H=0$. Therefore $D e_i =m_i < 0$, which implies that $e_i$ is contained in $\Bs|D|$ and, since $e_i$ spans the whole $E_i$, we get the claim.
\end{remark}

\subsection{A quadratic form}
%Let $\pi: X\to\pp^n$ be the blow-up of $\pp^n$ at $r$ points in very general position $p_1,\dots,p_r$ with exceptional divisors $E_1,\dots,E_r$ and let $H$ be the pull-back of a hyperplane of $\pp^n$. 
Consider the quadratic form on $\Pic(X)$ whose matrix with respect to the basis $H$, $E_1,\dots,E_r$ is diagonal with 
\[
  H^2=n-1
  \qquad
  E_1^2=\dots=E_r^2=-1.
\]
From now on $D_1\cdot D_2$ will denote the value of the corresponding bilinear form defined by the quadratic form. Observe that the lattice $(\Pic(X),\cdot)$ has determinant $\pm(n-1)$, so that it is unimodular if and only if $n=2$ in which case it coincides with the Picard lattice of $X$. 

\begin{definition}
Let $R\in\Pic(X)$ with $R^2=-2$. The {\em reflection} defined by $R$ is the $\zz$-linear map:
\[
  \sigma_R:\Pic(X)\to\Pic(X)
  \qquad
  D\mapsto D+(D\cdot R)R.
\]
\end{definition}

Observe that $\sigma_R$ is the reflection in $\Pic(X)$ with respect to the hyperplane orthogonal to $R$.
We will denote by 
\[
  F:=H-E_1-\dots-E_{n+1}
  \qquad
  F_i:=E_i-E_{i+1},\text{ for }1\leq i\leq r-1
\]
\begin{definition}
We consider the following subgroups of $\zz$-linear isometries of $\Pic(X)$ defined by:
\[
  S(X) := \langle\sigma_i : 1\leq i\leq r-1\rangle
  \qquad
  W(X) := \langle\sigma, S(X)\rangle,
\]
where $\sigma_i:=\sigma_{F_i}$ and $\sigma:=\sigma_F$. 
\end{definition}

%\begin{remark}
%Let us consider the group $S_r$, of permutations on $r$ elements.
%Observe that the effect of applying $\sigma_k$ to a class $D=dH-\sum_i m_iE_i$ is to exchange the coefficients of $E_k$ and $E_{k+1}$. Therefore $\sigma_k$ corresponds to the 2-cycle $(k,k+1)$ and, since any permutation of $S_r$ is a composition of such cycles, we get that $S(X)\cong S_r$.

\begin{remark}
Observe that $S(X)\cong S_r$, the group of permutations on $r$ elements,
since $\sigma_k$ corresponds to the transposition $(k,k+1)$. The group $W(X)$ is not necessarily finite. The class $K:=\frac{1}{n-1}K_X$ is $W(X)$-invariant with $K\cdot F=K\cdot F_i=0$ for any $i$. Thus, since the quadratic form has signature $(1,r-1)$, in case $K^2=n+3+\frac{4}{n-1}-r>0$, the restriction to $K^{\bot}$ is negative definite, and $W(X)$ is the Weyl group of the lattice $(K^{\bot},\cdot)$. The following table describes $(K^{\bot},\cdot)$ for all the values of $n$ and $r$ such that $K^2>0$ (see~\cite{DV}, \cite{Do} and \cite{Mu} for a detailed discussion of these lattices).
Observe that a set of simple roots for the lattice $(K^{\bot},\cdot)$ is always given by $F$ and the $F_i$'s.

\begin{center}
\begin{tabular}{r|r|l}
$n$ & $r$ & $K^{\bot}$\\
\hline
$\geq 2$ & $\leq n+2$ & $A_r$\\
$\geq 2$ & $n+3$ & $D_{n+3}$\\
4 & 8 & $E_8$\\
3 & 7 & $E_7$\\
2 & 6 & $E_6$\\
2  & 7 & $E_7$\\
2  & 8 & $E_8$
\end{tabular}
\end{center}
\end{remark}
%\end{remark}

%\begin{proposition}\label{elementary}
%Let $F:=H-E_1-\dots-E_{n+1}$ and let $\sigma$ be defined as before. Then, for any divisor $D$ on $X$ we have:
%\[
%  \sigma_*(D)=D+(D\cdot F)F.
%\]
%\end{proposition}

%Observe that Proposition~\ref{elementary} shows that $\sigma$ induces a reflection in $\Pic(X)$ with respect to the hyperplane $F^{\bot}$ orthogonal to $F$. This because $F^2=-2$ so that $\sigma_*(F)=-F$.
%Consider the vectors $F_i:=E_i-E_{i+1}$ and the associated reflections $\phi_i\in\Aut(\Pic(X))$ defined by $\phi_i(D):=D+(D\cdot F_i)F_i$.

\section{Small  modifications}
The aim of this section is to relate the elements of $W(X)$ with some birational maps of $X$. In order to do that we first recall the following definition. 

\begin{definition}\label{def:sqm}
A {\em small  modification} $\varphi: X_1\dashrightarrow X_2$ is a birational map which is an isomorphism in codimension 1, i.e. there exist open subsets $U_i\subseteq X_i$, such that ${\rm codim} (X_i\setminus U_i)\geq 2$ and $\varphi_{|U_1}: U_1\to U_2$ is an isomorphism.

%We denote by $\sqm(X)$ the group of small  modifications of $X$.
Given a divisor $D\subseteq X_1$ one defines the isomomorphism:
\[
  \varphi_*:\Div(X_1)\to\Div(X_2)
  \qquad
  D\mapsto\overline{\varphi(D\cap U_1)}.
\]
We will denote by the same symbol $\varphi_*$ the induced isomomorphism $\Pic(X_1)\to\Pic(X_2)$.
\end{definition}

An immediate consequence of Definition~\ref{def:sqm} is the following.
\begin{proposition}\label{sqm}
Let $\varphi: X_1\dashrightarrow X_2$ be a small  modification. Then $h^0(\varphi_*(D))\allowbreak =h^0(D)$ for any $D\in\Pic(X_1)$.
\end{proposition}

Let us go back now to the blow-up $\pi : X\rightarrow \pp^n$ at $r$ points $p_1,\dots,p_r$ in very general position, with $r\geq n+1$. 
We can suppose that the first $n+1$ points are the fundamental ones. 
%and $p_{n+2}=(1:1:\dots:1)$. 
Consider the small  modification $\varphi: X\dashrightarrow X'$ induced by the birational map:
\[
  \phi:\pp^n\dashrightarrow\pp^n
  \qquad
  (x_0:\dots :x_n)\mapsto (x_0^{-1}:\dots:x_n^{-1}),
\]
where $\pi': X'\rightarrow\pp^n$ is the blow-up of $\pp^n$ at $p_i'=p_i,$ for $i\leq n+1$, and $p_k'=\phi(p_k)$ for $k>n+1$.
Observe that by choosing $\{p_1,\dots,p_r\}\in\mathcal{U}\cap\phi(\mathcal{U})$ 
one has $\{p_1',\dots,p_r'\}\in\mathcal{U}\cap\phi(\mathcal{U})$ (where $\mathcal{U}$ is as in~\eqref{U}) so that the $p_i'$ are still in very general position.
Therefore, even if $X$ and $X'$ are not isomorphic, we can (and from now on we do)
identify $\Pic(X')$ and $\Pic(X)$, so that $\varphi_*$ can be considered as a $\zz$-linear map on $\Pic(X)$.

\begin{proposition}\label{elementary}
With the same notation as above we have $\varphi_*=\sigma$.
\end{proposition}
\begin{proof}
Recall that $F=H-E_1-\dots-E_{n+1}$ and that $\sigma(D)=D+(D\cdot F)F$. Since the point $p_i$, with $i\leq n+1$, is mapped by $\varphi$ to the hyperplane $x_{i+1}=0$, then
$\varphi_*(E_i)= E_i+F=\sigma(E_i)$.
Moreover, from $p_k'=\varphi(p_k)$, we get $\varphi_*(E_k)=E_k=\sigma(E_k)$ for $k>n+1$.
On the other hand, $\varphi$ maps the hyperplane $x_0=0$ to $p_1$, since it is an involution. Thus $\varphi_*(F+E_1)=E_1=\sigma(F+E_1)$.
We conclude observing that $E_1,\dots,E_r$, $F+E_1$ form a basis of $\Pic(X)$.\end{proof}
%We conclude this section proving the following proposition saying that the $h^0$ and the product are preserved by $W(X)$: 
\begin{proposition}\label{preserve}
Let $D,D'\in\Pic(X)$; then for any $w\in W(X)$: 
\begin{enumerate}
\item $w(D)\cdot w(D') = D\cdot D'$;
\item $h^0(w(D))=h^0(D)$, moreover $D$ is integral if and only if $w(D)$ is.
\end{enumerate}
\end{proposition}
\begin{proof}
The first statement follows from the fact that any $w\in W(X)$ is a composition of isometries of $(\Pic(X),\cdot)$.
For the second statement observe that since the points $p_1,\dots,p_r$ of $\pp^n$ are in very general position and $\sigma_k$ exchanges $p_k$ with $p_{k+1}$, then $h^0(\sigma_k(D))=h^0(D)$ and moreover $D$ is integral if and only if $\sigma_k(D)$ is integral. 
Observe that $h^0(D)=h^0(\varphi_*(D))=h^0(\sigma(D))$, by Propositions~\ref{sqm} and~\ref{elementary}. Moreover, since $\varphi$ is an isomorphism on $U\subseteq X$, with ${\rm codim}(X\setminus U)\geq 2$, then $\overline{\varphi(D\cap U)}$ is integral if and only if $D$ is integral. This completes the proof, by Proposition~\ref{elementary}.
\end{proof}

\section{Classes in standard form}
In this section, given a class $D\in\Pic(X)$ we find a representative $D'$ in the orbit $W(X)\cdot D$ which we will call in pre-standard form (see~\cite{LaUg1} and \cite{Du2}).
We will see in the following sections that these objects play an important rule in the formulation of conjectures for special divisors in the blow up of $\pp^2$ and $\pp^3$.

\begin{definition}
A class $D:=dH-\sum_im_iE_i$ is in {\em pre-standard form} if one of the following equivalent conditions holds:
\begin{enumerate}
\item $D\cdot (H-(n-1)E_1)\geq 0$, $D\cdot F_i\geq 0$ and $D\cdot F\geq 0$, for any $i=1,\dots,r-1$;
\item $d\geq m_1\geq\dots\geq m_r$ and $(n-1)d\geq m_1+\dots+m_{n+1}$.
\end{enumerate}
If in addition $D\cdot E_r\geq 0$, or equivalently $m_r\geq 0$, then $D$ is in {\em standard form}.
\end{definition}

\begin{proposition}\label{pre}
Let $D\in\Pic(X)$ be an effective class. Then there exists a $w\in W(X)$ such that
$w(D)$ is in pre-standard form.
\end{proposition}
\begin{proof}
Write $D:=dH-\sum_im_iE_i$ and observe that $d\geq m_1$ since $h^0(D)>0$. We proceed by induction on $d\geq 0$. If $d=0$, then $m_i\leq 0$ for any $i$, 
%since $h^0(D)>0$. 
and applying an element of $S(X)$ we obtain a divisor $D'$ in pre-standard form.
Assume now that $d>0$ and that the statement is true for $d'<d$. After applying an element of $S(X)$ we can assume that $m_1\geq\dots\geq m_r$. If $D\cdot F<0$, then $\sigma(D)=d'H-\sum_im_i'E_i$ with $d'=d+D\cdot F < d$. By induction hypothesis there exists a $w'\in W(X)$ such that $w'(\sigma(D))$ is pre-standard. By taking $w:=w'\circ \sigma$ we get the thesis.
\end{proof}

\subsection{An algorithm for the pre-standard form}
The Maple program  {\tt std} is part of the package {\tt StdClass} that can be freely downloaded (see~\cite{LaUg2}). Given a class $D:=dH-\sum_im_iE_i$, it returns its pre-standard form $D'=d'H-\sum_im_i'E_i$. Here $n=\dim(X)$.
\begin{enumerate}\leftskip 15mm
\item[INPUT =] $n,[d,m_1,\dots,m_r]$.
\item[OUTPUT =] $[d',m_1',\dots,m_r']$.
\end{enumerate}
%The structure of the program is the following
%\begin{verbatim}
%while(D is not pre-standard and deg(D)>=0)
%         {sort the multiplicities of D;
%         if DF<0 then D:=D+(DF)F;}
%return(D)         
%\end{verbatim}
Here is a Maple session.
\begin{verbatim}
> with(StdClass):
> std(3,[4,3,3,3,3]);
  [0, -1, -1, -1, -1]
\end{verbatim}

\section{(-1)-Classes}
We are now going to introduce some other particular classes in $\Pic(X)$, i.e. the {\em (-1)-classes}, which turn out to be a generalization of $(-1)$-curves of $\pp^2$. Next we analyze the relation between classes in standard form and $(-1)$-classes.  
\begin{definition}
A $(-1)$-class $E$ is an integral class with $h^0(E)>0$ such that $E^2=E\cdot K=-1$, where $K:=\frac{1}{n-1}K_X$.
\end{definition}
Observe that each $E_i$ is a $(-1)$-class and that if $n=2$, then $(-1)$-classes correspond to $(-1)$-curves.
\begin{lemma}\label{-1deg}
Let $E$ be a $(-1)$-class such that $E\cdot F_i\geq 0$ and $E\cdot E_i\geq 0$. Then $E\cdot F<0$.
\end{lemma}
\begin{proof}
Let $E:=dH-\sum_im_iE_i$. By hypothesis the multiplicities are in a decreasing order and $m_{n+1}\geq 0$. Therefore $E^2+m_{n+1}E\cdot K<0$, or equivalently
\[
  d((n-1)d-m_{n+1}(n+1))-\sum_{i=1}^{n+1}m_i(m_i-m_{n+1})<
  \sum_{i=n+2}^rm_i(m_i-m_{n+1}).
\]
Since the right hand side of the inequality is non-positive, the left hand side is negative. Writing $(n-1)d=\sum_{i=1}^{n+1}m_i+E\cdot F$ and substituting one obtains
\[
  \sum_{i=1}^{n+1}(d-m_i)(m_i-m_{n+1})+d(E\cdot F)<0
\]
which implies the thesis. 
\end{proof}

\begin{proposition}\label{-1}
A class $E$ is a $(-1)$-class if and only if there exists a $w\in W(X)$ such that
$E=w(E_1)$. 
\end{proposition}
\begin{proof}
If $w\in W(X)$, then $w(E_1)^2=w(E_1)\cdot K=-1$, since $w$ is an isometry and $w(K)=K$. Moreover $w(E_1)$ is integral, by Proposition~\ref{preserve}, so that $w(E_1)$ is a $(-1)$-class.

Assume now that $E:=dH-\sum_im_iE_i$ is a $(-1)$-class.
Modulo an element of $S(X)$ we can assume that the multiplicities are in a decreasing order, or equivalently $E\cdot F_i\geq 0$. Moreover, if $m_i<0$, then $E_i\subseteq\Bs|E|$, by Remark \ref{mi<0}, so that $E=E_i$ since $E$ is integral. Thus $w(E)=E_1$, where $w\in W(X)$ is the transposition $(1,i)$. 
We now proceed by induction on $m_{n+1}\geq -1$. We have already proved the first step of the induction.
Assume the property is true for any $m_{n+1}<m$ and let us prove it for $m_{n+1}=m\geq 0$.
By Lemma~\ref{-1deg} we have $E\cdot F\leq 0$.
%$E^2+m_{n+1}E\cdot K<0$ or equivalently
%\[
%  d((n-1)d-m_{n+1}(n+1))-\sum_{i=1}^{n+1}m_i(m_i-m_{n+1})<
%  \sum_{i=n+2}^rm_i(m_i-m_{n+1}).
%\]
%Since the right hand side of the inequality is non-positive, then the left hand side is negative. Writing $(n-1)d=\sum_{i=1}^{n+1}m_i+E\cdot F$ and substituting one obtains
%\[
%  \sum_{i=1}^{n+1}(d-m_i)(m_i-m_{n+1})+d(E\cdot F)<0
%\]
%which implies $E\cdot F<0$. 
Thus $\sigma(E)=d'H-\sum_im_i'E_i$, with $m_{n+1}'=m_{n+1}+E\cdot F<m$ and we conclude by induction.
\end{proof}

%The following corollary is a direct consequence of the proof of Proposition~\ref{-1}.

%\begin{corollary}\label{-1deg}
%Let $E$ be a $(-1)$-class such that $E\cdot F_i\geq 0$ and $E\cdot E_i\geq 0$. Then $E\cdot F<0$.
%\end{corollary}

We are now ready for the main theorem of this section.

\begin{theorem}\label{standard}
Let $D$ be in standard form. Then $w(D)\cdot E\geq 0$ for any $(-1)$-class $E$ and any $w\in W(X)$.
\end{theorem}
\begin{proof}
Let $E=dH-\sum_im_iE_i$. We first prove by induction on $d\geq 0$ that $D\cdot E\geq 0$.
If $d=0$, then $E=E_i$ for some $i$. Since $D$ is in standard form, then $D\cdot E_i\geq 0$. 
If $d>0$, let $E'$ be the divisor obtained from $E$ by reordering the multiplicities in decreasing order, then $D\cdot E'\leq D\cdot E$ and $E'$ is a $(-1)$-divisor which satisfies $E'\cdot F_i\geq 0$. Thus we can assume that both $E\cdot F_i\geq 0$ and $E\cdot E_i\geq 0$ are satisfied. By Lemma~\ref{-1deg} we have $E\cdot F<0$. Hence
\[
  D\cdot (\sigma(E)-E)=D\cdot (E\cdot F)F\leq 0
\]
since $D\cdot F\geq 0$. Let $\sigma(E)=d'H-\sum_im_i'E_i$, with $d'=d+E\cdot F<d$. By induction hypothesis we have $D\cdot \sigma(E)\geq 0$ which implies $D\cdot E\geq 0$.\\
In order to conclude the proof, let $w'\in W(X)$ be the inverse of $w$. Then $w(D)\cdot E = D\cdot w'(E)\geq 0$ since $w'(E)$ is a $(-1)$-class by Proposition~\ref{-1}.
%This is obvious for any $w\in S(X)$, while if $w=\sigma$ we have
%\[
%  (\sigma(D)-D)\cdot E=(D\cdot F)F\cdot E\leq 0.
%\]
\end{proof}

The following corollary shows that some geometric properties of $(-1)$-curves on a rational surface generalize to $(-1)$-classes. 

\begin{corollary}
Let $D\in\Pic(X)$ be an effective class and let $E$, $E'$ be $(-1)$-classes. 
\begin{enumerate}
\item If $D\cdot E<0$, then $E\subseteq\Bs|D|$. 
\item If $E$, $E'$ both have negative product with $D$, then $E\cdot E'=0$.
\end{enumerate}
\end{corollary}
\begin{proof}
Assume that $D\cdot E<0$ and let $w\in W(X)$ be such that $w(D)$ is in pre-standard form, by Proposition~\ref{pre}.
Write $w(D)=M+\sum_ia_iE_i$, where $M$ is in standard form and $a_i>0$ for any $i$. 
Observe that $w(D)\cdot w(E)=D\cdot E<0$ by Proposition \ref{preserve}, and $M\cdot w(E)\geq 0$ by Theorem~\ref{standard}. This implies that $w(E)\cdot\sum_ia_iE_i<0$, and in particular $w(E)\cdot E_i<0$ for some $i$. Since $w(E)$ is integral and $E_i\subseteq\Bs|w(E)|$, by Remark~\ref{mi<0}, then $w(E)=E_i$. Still, by Remark~\ref{mi<0}, we get $w(E)\subseteq\Bs|w(D)|$, which proves the first statement.

For the second statement observe that reasoning as above, we get $w(E)=E_j$, $w(E')=E_k$. Thus we conclude observing that $E\cdot E'=E_j\cdot E_k=0$.

\end{proof}

\section{Classes on the blow-up of $\pp^2$}

The aim of this section is to prove the equivalence between two conjectures about  special classes $D\in\Pic(X)$ on the blow-up $X$ of $\pp^2$ at points in very general position. We will provide a Maple program for calculating $h^0(D)$, based on the second conjecture.
We recall that a class $D:=dH-\sum_im_iE_i$ of $X$ is in standard form if
\[
  d\geq m_1\geq\dots\geq m_r\geq 0
  \qquad
  d\geq m_1+m_2+m_3.
\]

\begin{theorem}\label{equiv}
Let $X$ be the blow-up of $\pp^2$ at a finite number of points in very general position. The following two statements are equivalent:
\begin{enumerate}
\item an effective class $D$ is special if and only if there exists a $(-1)$-curve $E$ such that $D\cdot E\leq -2$;
\item an effective class in standard form is non-special.
\end{enumerate}
\end{theorem}
\begin{proof}
Let us first prove that (1) $\Rightarrow$ (2). If $D$ is an effective class in standard form, then $D\cdot E\geq 0$ for any $(-1)$-curve $E$, by Theorem~\ref{standard}. Hence $D$ is non-special.

We now prove (2) $\Rightarrow$ (1). Let $D$ be an effective divisor such that $D\cdot E\geq -1$ for any $(-1)$-curve $E$. Observe that if $D\cdot E=-1$, then $h^0(D)=h^0(D-E)$ and $\chi(D)=\chi(D-E)$, by the Riemann-Roch theorem. Thus we can assume that $D\cdot E\geq 0$ for any $(-1)$-curve $E$.
By Proposition~\ref{pre} there is a $w\in W(X)$ such that $D':=w(D)$ is in pre-standard form. Write $D'=d'H-\sum_im_i'E_i$ and observe that $m_i'\geq 0$ by our last assumption. Thus $D'$ is standard and hence non-special.
\end{proof}
We recall that statement (1) is known in literature as the {\em Segre, \allowbreak  Harbourne, \allowbreak Gimigliano, \allowbreak Hirschowiz conjecture}, or simply {\em S.H.G.H. conjecture} (see~\cite{Gi}, \cite{Ha}, \cite{Hi} and \cite{Se}), and it has been checked in a number of cases. The equivalence between the Segre conjecture and part (1) of Theorem~\ref{equiv} has been proved in~\cite{CiMi}.

\subsection{An algorithm for calculating $h^0(D)$}
The Maple program {\tt dim2} (see~\cite{LaUg2}), given $D:=dH-\sum_im_iE_i$, returns $h^0(D)$, assuming one of the two statements of Theorem~\ref{equiv} to be true. 
\begin{enumerate}\leftskip 15mm
\item[INPUT =] $[d,m_1,\dots,m_r]$.
\item[OUTPUT =] $h^0(D)$.
\end{enumerate}
Here is a Maple session.
\begin{verbatim}
> with(StdClass):
> dim2([96,34,34,34,34,34,34,34,34]);
  1
\end{verbatim}

\section{Classes on the blow-up of $\pp^3$}
The aim of this section is to state a conjecture about special classes $D\in\Pic(X)$ on the blow-up $X$ of $\pp^3$ at points in very general position. We will provide a Maple program for calculating $h^0(D)$, based on this conjecture. We recall that a class $D:=dH-\sum_im_iE_i$ of $X$ is in standard form if
\[
  d\geq m_1\geq\dots\geq m_r\geq 0
  \qquad
  2d\geq m_1+m_2+m_3+m_4.
\]

Let $Q$ be the strict transform of the quadric through the first $9$ points, or equivalently its class is $2H-E_1-\dots-E_9$. In what follows we assume that $D$ is in standard form. We wish to provide a criterion for deciding when $Q\subseteq\Bs(|D|)$. 
%Consider the exact sequence of sheaves
%\[
%\xymatrix{
%0\ar[r] & \Osh_X(D-Q)\ar[r] & \Osh_X(D)\ar[r] & \Osh_Q(D)\ar[r] & 0.
%}
%\]
\begin{proposition}
The divisor $D_{|Q}$ has non-negative intersection with any $(-1)$-curve of $Q$. 
\end{proposition}
\begin{proof}
Let $f_1$, $f_2$ be the pull-back of the classes of two rulings of the quadric and let $e_1,\dots,e_9$ be the nine exceptional curves. These classes form a basis of $\Pic(Q)$. Observe that $Q$ is the blow-up of $\pp^2$ at $10$ points, with basis of the Picard group given by:
\[
  f_1+f_2-e_1,\quad f_1-e_1, \quad f_2-e_1, \quad e_2,\dots, e_9.
\]
Since $e_i={E_i}_{|Q}$ and $f_1+f_2=H_{|Q}$, then the class $z$ of $D_{|Q}$ has degree and multiplicities with respect to this basis given by:
\[
  2d-m_1,\quad d-m_1,\quad d-m_1,\quad m_2,\dots,m_9.
\]
If $m_4\leq d-m_1<m_3$, then either $z$ or $\sigma(z)$ is in standard form after reordering the multiplicities in decreasing order.
In the remaining cases $z$ is in standard form after reordering the multiplicities in decreasing order.
We conclude by Theorem~\ref{standard}.

%\[
%  z\cdot (f_1+f_2-e_1-(f_1-e_1))\geq 0
%  \quad 
%  z\cdot (f_1-f_2)=0
%\] 

\end{proof}

Thus, assuming the S.H.G.H. conjecture to be true for 10 points of $\pp^2$ in very general position, we deduce that $D_{|Q}$ is non-special, or equivalently that
$h^0(D_{|Q})\cdot h^1(D_{|Q})=0$. 
We define $q(D):=\chi(D_{|Q})$ and observe that
\[
  q(D)=(d+1)^2-\frac{1}{2}\sum_{i=1}^9m_i(m_i+1).
\]

\begin{proposition}
Assume that the S.H.G.H. conjecture is true for 10 points of $\pp^2$ in very general position. If $q(D)\leq 0$, then $h^0(D)=h^0(D-Q)$.
\end{proposition}
\begin{proof}
By Riemann-Roch, Serre's duality and the fact that $D_{|Q}$ is non-special, we deduce $h^0(D_{|Q})=0$, so that $h^0(D-Q)=h^0(D)$ which proves the thesis.
\end{proof}

The following conjecture has been formulated for the first time in~\cite{LaUg1}.

\begin{conjecture}\label{con1}
Let $X$ be the blow-up of $\pp^3$ at a finite number of points in very general position and let $D:=dH-\sum_im_iE_i$ be in standard form.
\begin{enumerate}
\item If $q(D)\leq 0$, then $h^0(D)=h^0(D-Q)$.
\item If $q(D) > 0$, then $D$ is special if and only if $d<m_1+m_2-1$ and 
\[
  h^0(D)=\binom{d+3}{3}-\sum_{i=1}^r\binom{m_i+2}{3}
  +\sum_{m_i+m_j>d+1} \binom{m_i+m_j-d+1}{3}.
\]
\end{enumerate}
\end{conjecture}
Conjecture~\ref{con1} has been proved for $r\leq 8$ and any multiplicities (see~\cite{DVLa}), for $m_i\leq 4$ and any $r$ (see~\cite{BaBr} and~\cite{Du1}).
Observe that if $m_2+m_3>d+1$, then $m_1+m_4<d-1$ since $D$ is in standard form. Hence in this case the sum on the right hand side of (2) is on the pairs $(m_1,m_2)$, $(m_1,m_3)$, $(m_2,m_3)$. If $m_2+m_3<d+1$,then the sum is over all the pairs $(m_1,m_i)$, such that $m_1+m_i>d+1$.

\subsection{An algorithm based on Conjecture~\ref{con1}(1)}
The Maple program {\tt quad} (see~\cite{LaUg2}), given $D:=dH-\sum_im_iE_i$ returns a standard class $D'=d'H-\sum_im_i'E_i$ with $h^0(D')=h^0(D)$ and $q(D')>0$.
\begin{enumerate}\leftskip 15mm
\item[INPUT =] $[d,m_1,\dots,m_r]$.
\item[OUTPUT =] $[d',m_1',\dots,m_r']$.
\end{enumerate}

Here is a Maple session.
\begin{verbatim}
> with(StdClass):
> quad([19,9,9,9,9,9,9,9,9,9]);               
  [15, 7, 7, 7, 7, 7, 7, 7, 7, 7]
\end{verbatim}

\subsection{An algorithm for calculating $h^0(D)$}

The Maple program {\tt dim3} (see~\cite{LaUg2}), given $D:=dH-\sum_im_iE_i$, returns $h^0(D)$ according to Conjecture~\ref{con1}. It makes use of the function {\tt quad}.
\begin{enumerate}\leftskip 15mm
\item[INPUT =] $[d,m_1,\dots,m_r]$.
\item[OUTPUT =] $h^0(D)$.
\end{enumerate}
Here is a Maple session.
\begin{verbatim}
> with(StdClass):
> dim3([19,9,9,9,9,9,9,9,9,9]);
  60
\end{verbatim}

\section{examples}
In this section we consider several examples of effective classes $D\in\Pic(X)$, where $X$ is the blow-up of $\pp^n$ at points in very general position. Denote by 
\[
  L_3(d;m_1^{a_1},\dots,m_s^{a_s})
\]
the class of the strict transform of a hypersurface of degree $d$ through $a_i$ points of multiplicity $\geq m_i$, for $i=1,\dots, s$.

\subsection{Pre-standard form}
Consider the class $D:=L_n(n+1;n^{n+1})$. Then $\sigma(D)=L_n(0;(-1)^{n+1})$ is in pre-standard form. This proves that $D$ is sum of $(-1)$-classes.
Here we run a Maple test when $n=5$.
\begin{verbatim}
> Std(5,[6,[5,6]]);
  [0, [-1, 6]]
\end{verbatim}

\subsection{Dimension $2$}
According to the S.H.G.H. conjecture or its equivalent statement given in Theorem~\ref{equiv}(2), the class $L_2(d;m^r)$ is non-special if $d\geq 3m$. If $d<3m$ and $r>8$, then it is non-effective since its pre-standard form has negative degree. If 
$r\leq 8$, then we can have a special class, like for example $D:=L_2(96;34^8)$. We expect $D$ to be non-effective but we have
\begin{verbatim}
> Dim2([96,[34,8]]); 
  1
\end{verbatim}
As shown by the following
\begin{verbatim}
> Std(2,[96,[34,8]]);
  [0, [-2, 8]]
\end{verbatim}
we have that $w(D)=\sum_i2E_i$ so that $D$ is a sum of $(-1)$-classes as well.

\subsection{Dimension $3$}
The class $L_3(d;m^r)$ is in standard form if $d\geq 2m\geq 0$. In this case, according to Conjecture~\ref{con1}, it is non-special if $(d+1)^2-\frac{9}{2}m(m+1)>0$
or equivalently if
\[
  d > -1+\frac{3\sqrt{2m^2+2m}}{2}.
\]
An example of a class $D$ in standard form with $q(D)\leq 0$ is $L_3(2m+1;m^9)$,
with $m\geq 8$. In this case $h^0(D)=60$ does not depend on $m$, even if the expected dimension does.
\begin{verbatim}
> Quad([19,[9,9]]);                                      
  [15, [7, 9]]
\end{verbatim}
The expected dimension of this class is $55$, while we have:
\begin{verbatim}
> Dim3([19,[9,9]]);
  60
\end{verbatim}

\end{document}